\documentclass[11pt]{article}
\usepackage{amsmath,amssymb,amsthm,amscd}

\newtheorem{theorem}{Theorem}[section]
\newtheorem{lemma}[theorem]{Lemma}
\newtheorem{proposition}[theorem]{Proposition}

\theoremstyle{plain}

\theoremstyle{definition}
\newtheorem{definition}[theorem]{Definition}
\newtheorem{remark}[theorem]{Remark}

\numberwithin{equation}{section}

\renewcommand{\labelenumi}{\textup{(\theenumi)}}

\title{One-sided topological conjugacy 
of topological Markov shifts,
continuous full groups and Cuntz--Krieger algebras
}
\author{Kengo Matsumoto \\
Department of Mathematics \\
Joetsu University of Education \\
Joetsu, 943-8512, Japan
}

\begin{document}

\date{}

\maketitle

\def\det{{{\operatorname{det}}}}

\begin{abstract}
We will characterize topological conjugate one-sided topological Markov shifts in
terms of their subgroups of continuous full groups and subalgebras of Cuntz--Krieger algebras. 
\end{abstract}

{\it Mathematics Subject Classification}:
 Primary 37A55; Secondary  37B10, 46L35.

{\it Keywords and phrases}:
Topological Markov shifts, continuous orbit equivalence,  Cuntz--Krieger algebras, 
topological conjugacy.

\newcommand{\Ker}{\operatorname{Ker}}
\newcommand{\sgn}{\operatorname{sgn}}
\newcommand{\Ad}{\operatorname{Ad}}
\newcommand{\ad}{\operatorname{ad}}
\newcommand{\orb}{\operatorname{orb}}

\def\Re{{\operatorname{Re}}}
\def\det{{{\operatorname{det}}}}
\newcommand{\K}{\mathbb{K}}

\newcommand{\N}{\mathbb{N}}
\newcommand{\C}{\mathbb{C}}
\newcommand{\R}{\mathbb{R}}
\newcommand{\Rp}{{\mathbb{R}}^*_+}
\newcommand{\T}{\mathbb{T}}
\newcommand{\Z}{\mathbb{Z}}
\newcommand{\Zp}{{\mathbb{Z}}_+}
\def\AF{{{\operatorname{AF}}}}

\def\OA{{{\mathcal{O}}_A}}
\def\OAf{{{\mathcal{O}}_{A_f}}}
\def\OB{{{\mathcal{O}}_B}}
\def\OTA{{{\mathcal{O}}_{\tilde{A}}}}
\def\SOA{{{\mathcal{O}}_A}\otimes{\mathcal{K}}}
\def\SOB{{{\mathcal{O}}_B}\otimes{\mathcal{K}}}
\def\SOTA{{{\mathcal{O}}_{\tilde{A}}\otimes{\mathcal{K}}}}
\def\F{{\mathcal{F}}}
\def\FA{{{\mathcal{F}}_A}}
\def\PA{{{\mathcal{P}}_A}}
\def\FAf{{{\mathcal{F}}_{A,f}}}
\def\FAg{{{\mathcal{F}}_{A,g}}}
\def\FBg{{{\mathcal{F}}_{B,g}}}
\def\FAb{{{\mathcal{F}}_A^b}}
\def\FB{{{\mathcal{F}}_B}}
\def\DA{{{\mathcal{D}}_A}}
\def\DB{{{\mathcal{D}}_B}}
\def\DZ{{{\mathcal{D}}_Z}}
\def\DTA{{{\mathcal{D}}_{\tilde{A}}}}
\def\Ext{{{\operatorname{Ext}}}}
\def\Max{{{\operatorname{Max}}}}
\def\Per{{{\operatorname{Per}}}}
\def\PerB{{{\operatorname{PerB}}}}
\def\Homeo{{{\operatorname{Homeo}}}}
\def\HomeoCOE{{{\operatorname{H}_{\operatorname{coe}}}}}
\def\HA{{{\frak H}_A}}
\def\HB{{{\frak H}_B}}
\def\HSA{{H_{\sigma_A}(X_A)}}
\def\Out{{{\operatorname{Out}}}}
\def\Aut{{{\operatorname{Aut}}}}
\def\Ad{{{\operatorname{Ad}}}}
\def\Inn{{{\operatorname{Inn}}}}
\def\det{{{\operatorname{det}}}}
\def\exp{{{\operatorname{exp}}}}
\def\nep{{{\operatorname{nep}}}}
\def\sgn{{{\operatorname{sign}}}}
\def\cobdy{{{\operatorname{cobdy}}}}
\def\Ker{{{\operatorname{Ker}}}}
\def\ind{{{\operatorname{ind}}}}
\def\id{{{\operatorname{id}}}}
\def\supp{{{\operatorname{supp}}}}
\def\co{{{\operatorname{co}}}}
\def\scoe{{{\operatorname{scoe}}}}
\def\coe{{{\operatorname{coe}}}}
\def\I{{\mathcal{I}}}

\def\S{\mathcal{S}}

\def\tS{\tilde{S}}
\def\coe{{{\operatorname{coe}}}}
\def\scoe{{{\operatorname{scoe}}}}
\def\uoe{{{\operatorname{uoe}}}}
\def\ucoe{{{\operatorname{ucoe}}}}
\def\event{{{\operatorname{event}}}}

\section{Introduction}

The notion of continuous orbit equivalence between one-sided topological Markov shifts
is closely related to classification of not only Cuntz--Krieger algebras
but also a certain class of countable discrete non amenable groups called 
continuous full groups.
In this paper, we will further study 
close relations among topological conjugacy of one-sided topological Markov shifts,
subalgebras of Cuntz--Krieger algebras, subgroups of continuous full groups.
For an irreducible, non permutation matrix $A = [A(i,j)]_{i,j=1}^N$,
the one-sided topological Markov shift for the matrix $A$
 is defined by a topological dynamical system $(X_A, \sigma_A)$
 with its shift space 
 $X_A =\{ (x_n)_{n \in \N} \in \{1,\dots, N\}^\N
\mid A(x_n, x_{n+1}) =1, n \in \N\}$
 and its transformation
 $\sigma_A: X_A\longrightarrow X_A$ 
satisfying $\sigma_A((x_n)_{n\in \N}) = (x_{n+1})_{n \in \N},$
where the shift space $X_A$ is endowed with the relative topology of the infinite product topology
of $\{1,\dots, N\}^\N,$
so that $\sigma_A$ is a continuous surjection on the compact Hausdorff space $X_A$.
Two topological Markov shifts 
$(X_A,\sigma_A)$ and $(X_B,\sigma_B)$ are said to be {\it topologically conjugate}\/ if 
there exists a homeomorphism
$h: X_A\longrightarrow X_B$ such that $h \circ \sigma_A = \sigma_B \circ h$.
Generalizing the condition $h \circ \sigma_A = \sigma_B \circ h$,
we reach the definition of {\it continuous orbit equibalence}\/
between $(X_A,\sigma_A)$ and $(X_B,\sigma_B)$ (\cite{MaPacific})
that describes that 
there exist continuous functions
$k_1, l_1:X_A\longrightarrow \Zp$
and
$k_2, l_2:X_B\longrightarrow \Zp$
such that 
\begin{equation}\label{eq:coe12}
\begin{cases}
\sigma_B^{k_1(x)}(h(\sigma_A(x))) 
\thinspace & = \quad \sigma_B^{l_1(x)}(h(x)), \qquad x \in X_A,  \\
\sigma_A^{k_2(w)}(h^{-1}(\sigma_B(w))) 
\thinspace & = \quad \sigma_A^{l_2(w)}(h^{-1}(w)), \qquad w \in X_B,  
\end{cases}
\end{equation}
where $\Zp$ stands for the set of nonnegative integers.
If $k_1 \equiv 0, \, l_1 \equiv 1,\, k_2 \equiv 0, \, l_2 \equiv 1,$
the continuous orbit equivalence reduces to topological conjugacy.

The continuous full group $\Gamma_A$ for the matrix $A$ 
was introduced in \cite{MaPacific}
to study continuous orbit equivalence.
It is defined by the group of homeomorphisms $\tau$ on $X_A$
such that there exist continuous functions $k_\tau, l_\tau: X_A\longrightarrow \Zp$
satisfying 
\begin{equation} \label{eq:tau}
\sigma_A^{k_\tau(z)}(\tau(z)) = \sigma_A^{l_\tau(z)}(z), \qquad z \in X_A. 
\end{equation}
The group  was written as $[\sigma_A]_c$ in the earlier papers
\cite{MaPacific}, \cite{MaDCDS2013}.
It is a countable discrete non amenable group 
(see \cite{MaDCDS2013}, \cite{MMGGD}, \cite{MatuiCrelle}, etc.).
H. Matui studied such kinds of groups for more general setting 
in terms of \'etale groupoids (\cite{MatuiPLMS}, \cite{MatuiCrelle}).

The Cuntz--Krieger algebra for the matrix $A$
is defined to be the $C^*$-algebra $\OA$ generated by 
partial isometries $S_1,\dots, S_N$ satisfying the operator relations
$\sum_{j=1}^N S_j S_j^* =1,\, S_i^* S_i = \sum_{j=1}^N A(i,j) S_j S_j^*, i=1,\dots,N$
(\cite{CK}).
The $C^*$-algebra $\OA$ is simple, purely infinite 
and has an action $\rho^A$ of the circle group $\R/ \Z =\T$ defined by 
$\rho^A_t(S_j) = \exp(2 \pi \sqrt{-1} t) S_j$, 
$t \in  \T$.
The action is called the (standard) gauge action. 
Its fixed point algebra $\OA^{\rho^A}$ is known as an AF-algebra written 
$\F_A$. 
The original shift space $X_A$ is realized in $\OA$ as the commutative $C^*$-algebra
$C(X_A)$ of complex valued continuous functions on $X_A$ 
that is canonically isomorphic to a commutative $C^*$-subalgebra
$\DA$ of $\F_A$ generated by the projections of the form
$S_{\mu_1}\cdots S_{\mu_m}S_{\mu_m}^*\cdots S_{\mu_1}^*, 
(\mu_1,\dots,\mu_m) \in B_m(X_A)$
through the identification between 
$S_{\mu_1}\cdots S_{\mu_m}S_{\mu_m}^*\cdots S_{\mu_1}^* \in \DA$
and
$\chi_{U_\mu},$
where
$B_m(X_A)$ denotes the set of admissible words 
$\{(x_1,\dots, x_m) \in \{1,\dots, N\}^m \mid (x_n)_{n \in \N} \in X_A\}$
of length $m$
and 
$\chi_{U_\mu}$
is the characteristic function on $X_A$
of the cylinder set 
$U_\mu = \{ (x_n)_{n\in \N} \in X_A \mid x_1 = \mu_1,\dots, x_m = \mu_m\}$
for the word $\mu = (\mu_1,\dots,\mu_m).$
It is well known that the commutative $C^*$-subalgebra $\DA$ is a regular 
maximal abelian $C^*$-subalgebra of $\OA$ (\cite{CK}).

The following theorem proved in \cite{MaIsrael2015} shows a close relationship 
among the above three objects.

\begin{theorem}[{\cite[Corollary 1.2]{MaIsrael2015}, cf. \cite{MaPacific}, \cite{MMKyoto}, 
\cite{MatuiCrelle}}]
Let $A, B$ be irreducible non permutation matrices with entries in $\{0,1\}$.
The following assertions are mutually equivalent.
\begin{enumerate}
\renewcommand{\theenumi}{\roman{enumi}}
\renewcommand{\labelenumi}{\textup{(\theenumi)}}
\item The one-sided topological Markov shifts 
$(X_A,\sigma_A)$ and $(X_B,\sigma_B)$ are continuous orbit equivalent. 
\item There exists an isomorphism $\Phi:\OA\longrightarrow \OB$ of $C^*$-algebras 
satisfying  
$\Phi(\DA) = \DB$.
\item There exists an isomorphism
$\xi: \Gamma_A\longrightarrow \Gamma_B$ of groups.
\end{enumerate}
\end{theorem}
Several generalizations and related results to the above theorem 
are seen in \cite{CEOR}, \cite{CRS}, \cite{NO}, etc.

In \cite{MaPAMS2017}, the two notions of uniformly continuous orbit equivalence 
and eventual conjugacy of one-sided topological Markov shifts were introduced as 
strictly stronger equivalence relations than continuous orbit equivalence.
They were proved to be equivalent equivalence relations in \cite{MaPAMS2017}.
$(X_A,\sigma_A)$ and $(X_B,\sigma_B)$ are said to be {\it eventually conjugate}\/ 
if one may take the four continuous functions
$k_1, l_1, k_2, l_2$ in the definition \eqref{eq:coe12} 
of continuous orbit equivalence  as one constant integer.
In \cite{BC}, K. A. Brix and T. M. Carlsen 
found a pair of irreducible non permutation matrices 
such that their one-sided topological Markov shifts are 
eventually conjugate but not topologically conjugate.
Thanks to their paper, 
we know that  the notion of eventual conjugacy is strictly weeker than topological conjugacy.
Let $\Gamma^{\operatorname{AF}}_A$ be the subgroup of $\Gamma_A$ consisting of 
$\tau\in \Gamma_A$ such that $k_\tau(x)= l_\tau(x)$
for all $x \in X_A$ in \eqref{eq:tau}.
It is called the AF full group for $(X_A,\sigma_A)$. 
In \cite{MaPAMS2017}, the following characterization of eventual conjugacy was proved.
\begin{theorem}[{\cite[Theorem 1.5]{MaPAMS2017}}]
Let $A, B$ be irreducible non permutation matrices with entries in $\{0,1\}$.
The following assertions are mutually equivalent.
\begin{enumerate}
\renewcommand{\theenumi}{\roman{enumi}}
\renewcommand{\labelenumi}{\textup{(\theenumi)}}
\item 
$(X_A,\sigma_A)$ and $(X_B,\sigma_B)$ are eventually conjugate. 
\item There exists an isomorphism $\Phi:\OA\longrightarrow \OB$ of $C^*$-algebras 
satisfying 
$\Phi(\DA) = \DB$ and $\Phi\circ \rho^A_t = \rho^B_t \circ \Phi, t \in \T$.
\item There exists an isomorphism $\Phi:\OA\longrightarrow \OB$ of $C^*$-algebras 
satisfying 
$\Phi(\DA) = \DB$ and $\Phi(\F_A) = \F_B$.
\item There exists an isomorphism
$\xi: \Gamma_A\longrightarrow \Gamma_B$ of groups satisfying
$\xi(\Gamma^{\operatorname{AF}}_A) = \Gamma^{\operatorname{AF}}_B$.
\end{enumerate}
\end{theorem}
To obtain a characterization of topological conjugacy 
of one-sided topological Markov shifts in terms of 
Cuntz--Krieger algebras and continuous full groups,
we need more finer information than
their gauge actions and AF full groups
in the folowing way.
Let us denote by $C(X_A,\Z)$ the set of integer valued continuous 
functions on $X_A$.
For a function $f \in C(X_A, \Z)$, consider a one-parameter unitary 
group $\exp(2\pi\sqrt{-1} f t)$ in $\DA$. As 
$f$ is  a continuous function on $X_A$, 
it is regarded as an element of $\DA$.  
We then define gauge action $\rho^{A,f}$ on $\OA$ with potential function $f$ 
by the automorphism groups $\rho^{A,f}_t, t \in \T$ of $\OA$ satisfying
$\rho^{A,f}_t(S_j) =\exp(2\pi\sqrt{-1} f t) S_j, \, j=1,2,\dots,N.$     
If $f\equiv 1$, the generalized gauge action 
$\rho^{A,f}$ reduces to the standand gauge action.
Let us denote by 
$\F_{A,f}$ the fixed point algebra of $\OA$ under the action $\rho^{A,f}$.
It is called the {\it cocycle algebra for}\/ $f$.
For the counter part to the group $\Gamma^{\operatorname{AF}}_A$,
a subgroup of $\Gamma_A$ associated to $f$ 
was introduced in \cite{MaPre2020c}. 
It is called the cocycle full group and written 
$\Gamma_{A,f}$.
The subgroup
$\Gamma_{A,f}$ is defined in the following way.
For $(z,\tau) \in X_A\times \Gamma_A$, define
\begin{equation}
\rho^f(z,\tau) =\sum_{i=0}^{l_\tau(z)}f(\sigma_A^i(z)) -
\sum_{i=0}^{k_\tau(z)}f(\sigma_A^i(\tau(z))). 
\end{equation}
\begin{definition}[{\cite[Definition 1.3]{MaPre2020c}}]
The {\it cocycle full group for}\/ $f$ is defined by 
\begin{equation}
\Gamma_{A,f} = \{ \tau \in \Gamma_A \mid \rho^f(z,\tau) =0 \text{ for all } z \in X_A \}.
\end{equation}
\end{definition}
It is actually a  subgroup of $\Gamma_A$
(Proposition \ref{prop:group}).
If $f\equiv 1$, then the subgroup coincides with the AF full group,
that is,
$\Gamma_{A,1} = \Gamma^{\operatorname{AF}}_A$.
Let us stand for $U(\FAf,\DA)$
the unitary normalizer of $\DA$ in $\FAf$, that is,
\begin{equation}
U(\FAf,\DA) = \{ u \in U(\FAf) \mid u \DA u^* = \DA \}.
\end{equation} 
Then we will know that $\Ad(u)$ for $u \in U(\FAf,\DA)$
gives rise to an element of the cocycle  full group
$\Gamma_{A,f}$. 
Let us denote by $\Inn(\FAf,\DA)$ the group of inner automorphisms
of $\FAf$ keeping the subalgebra $\DA$ globally.
We will actually know that 
 there exists a short exaxt sequence 
\begin{equation}
1\longrightarrow U(\DA)
 \overset{\Ad}{\longrightarrow}\Inn(\FAf,\DA)
 \longrightarrow 
\Gamma_{A,f}
 \longrightarrow 1\label{eq:exactFAf1}
\end{equation}
that splits (Theorem \ref{thm:main1}).

As a main result of the paper, we will show the following theorem
that   
characterizes topological conjugate one-sided topological Markov shifts 
in terms of their cocycle full groups of continuous full groups and 
cocycle subalgebras of Cuntz--Krieger algebras. 
\begin{theorem}
\label{thm:mainthm}
Let $A, B$ be irreducible non permutation matrices with entries in $\{0,1\}$.
Then the following assertions are mutually equivalent.
\begin{enumerate}
\renewcommand{\theenumi}{\roman{enumi}}
\renewcommand{\labelenumi}{\textup{(\theenumi)}}
\item There exists a homeomorphism
$h:X_A\longrightarrow X_B$ such that 
$h \circ \sigma_A = \sigma_B\circ h.$
\item There exists an isomorphism $\Phi:\OA\longrightarrow \OB$ of $C^*$-algebras such that 
$\Phi(\DA) = \DB$ and 
\begin{equation*}
\Phi\circ \rho^{A, g\circ h}_t = \rho^{B,g}_t \circ \Phi
\qquad \text{ for all } g \in C(X_B,\Z), t \in \T,
\end{equation*}
where $h:X_A\longrightarrow X_B$ is a unique homeomorphism satisfying
$\Phi(f) = f\circ h^{-1}$ for $f \in \DA$.
\item
There exists an isomorphism $\Phi:\OA\longrightarrow \OB$ of $C^*$-algebras such that 
$\Phi(\DA) = \DB$ and 
\begin{equation*}
\Phi(\F_{A, g\circ h}) = \F_{B,g}\qquad
\text{ for all }
g \in C(X_B,\Z), t \in \T,
\end{equation*}
where $h:X_A\longrightarrow X_B$ is a unique homeomorphism satisfying
$\Phi(f) = f\circ h^{-1}$ for $f \in \DA$.
\item There exists an isomorphism
$\xi: \Gamma_A\longrightarrow \Gamma_B$ of groups such that 
\begin{equation*}
\xi(\Gamma_{A,g\circ h}) =\Gamma_{B,g}
\qquad
\text{  for all }g \in C(X_B,\Z),
\end{equation*}
where $h:X_A\longrightarrow X_B$
is a unique homeomorphism satisfying 
$\xi(\tau) = h\circ \tau\circ h^{-1}$ for $\tau \in \Gamma_A$.
\end{enumerate}
\end{theorem}
The equivalence between (i) and (ii) was proved in \cite{MaPre2020b}.
It is straightforward to see that the implication 
(ii) $\Longrightarrow$ (iii) holds.
The implication 
(iii) $\Longrightarrow$ (iv)
is due to Proposition \ref{prop:GammaUF}.
The implication (iv) $\Longrightarrow$ (i)
is a main ingredient in the paper,
that  will be proved in Section \ref{sec:final}.


\medskip

Throughout the paper, $\N, \Zp$ denote the set of positive integers, the set of nonnegative integers,
respectively.


\section{Cocycle full groups and cocycle algebras}
Throughout the section, we fix an irreducible non permutation 
matrix $A=[A(i,j)]_{i,j=1}^N$ with entries in $\{0,1\}$.

{\bf 1. Cocycle full groups.}

Recall that a homeomorphism $\tau: X_A\longrightarrow X_A$ 
belongs to the continuous full group $\Gamma_A$ if 
there exist $k_\tau, l_\tau \in C(X_A,\Zp)$ satisfying
\eqref{eq:tau}. 
Although the functions $k_\tau, l_\tau \in C(X_A,\Zp)$
are not uniquely determined for $\tau$  but its difference 
$d_\tau = l_\tau - k_\tau$ is unique (\cite[Lemma 7.6]{MMGGD}). 
For a function $f \in C(X_A,\Z)$ and $n \in \Zp$, let us define the function  
$f^n \in C(X_A,\Z)$ by
$f^n(x) = \sum_{i=0}^{n-1} f(\sigma_A^i(x))$
for $n\in \N$
and
$f^n(x)=0$ for $n=0$.
It is straightforward to see that identities 
$ 
f^{n+m}(x) = f^m(x) + f^n(\sigma_A^m(x))
$
for
$n,  m \in \Zp,\,\, x \in X_A
$
hold. 
Recall that  the function $\rho^f: X_A\times \Gamma_A\longrightarrow \Z$
is defined by
\begin{equation}\label{eq:rhofxtau}
\rho^f(x,\tau) =\sum_{i=0}^{l_\tau(x)}f(\sigma_A^i(x)) -
\sum_{i=0}^{k_\tau(x)}f(\sigma_A^i(\tau(x))), \qquad x \in X_A. 
\end{equation}
By the equality \eqref{eq:tau},
the identity
\begin{equation*}
\rho^f(x,\tau) =f^{l_\tau(x)}(x) -  f^{k_\tau(x)}(\tau(x)), \qquad x \in X_A
\end{equation*}
holds, because $f(\sigma_A^{l_\tau(x)}(x))=f(\sigma_A^{k_\tau(x)}(\tau(x))).$  
\begin{lemma}
$\rho^f(x,\tau) $ does not depend on the choice of $l_\tau, k_\tau$ as long as they satisfy
\eqref{eq:tau}.
\end{lemma}
\begin{proof}
For $\tau \in \Gamma_A$,
let $l^\prime_\tau, k^\prime_\tau$ be another pair of $l_\tau, k_\tau$ satisfying \eqref{eq:tau}. 
By \cite[Lemma 7.6]{MMGGD},
one knows $l_\tau(x) - k_\tau(x) =l^\prime_\tau(x) - k^\prime_\tau(x)$.
We put
$m_\tau(x) = l^\prime_\tau(x) - l_\tau(x) (=k^\prime_\tau(x) - k_\tau(x))$.
One may assume that $m_\tau(x) \ge 0$.
It then follows that 
\begin{align*}
 & f^{l^\prime_\tau(x)}(x) -  f^{k^\prime_\tau(x)}(\tau(x)) \\ 
= & f^{l_\tau(x) +m_\tau(x)}(x) -  f^{k_\tau(x)+m_\tau(x)}(\tau(x)) \\ 
= & \{f^{l_\tau(x)}(x) + f^{m_\tau(x)}(\sigma_A^{l_\tau(x)}(x)) \}
- \{ f^{k_\tau(x)}(\tau(x))  + f^{m_\tau(x)}(\sigma_A^{k_\tau(x)}(\tau(x)) \} \\ 
= & f^{l_\tau(x)}(x) -  f^{k_\tau(x)}(\tau(x)) 
\end{align*}
\end{proof}
Although the following lemma was shown in \cite{MaPre2020c},
we give its proof for the sake of completeness.
\begin{lemma}[{\cite[Lemma 2.5]{MaPre2020c}}] \label{lem:2.1}
For $ \tau, \tau_1, \tau_2 \in \Gamma_A$ and $x \in X_A$, 
we have the following identities. 
\begin{enumerate}
\renewcommand{\theenumi}{\roman{enumi}}
\renewcommand{\labelenumi}{\textup{(\theenumi)}}
\item
$
\rho^f(x,\tau_2 \circ \tau_1) 
=
\rho^f(x, \tau_1) + \rho^f(\tau_1(x), \tau_2).
$
\item
$
\rho^f(x,\tau^{-1}) 
= -\rho^f(x,\tau).$
\end{enumerate}
\end{lemma}
\begin{proof}
%
(i) By the equalities
$$
l_{\tau_2\circ \tau_1} = l_{\tau_1}  + l_{\tau_2} \circ \tau_1,
\qquad
k_{\tau_2\circ \tau_1} =k_{\tau_1} + k_{\tau_2} \circ \tau_1
$$
as in  \cite[Lemma 7.7]{MMGGD}, we have 
\begin{align*}
& \rho^f(x,\tau_2 \circ \tau_1) \\ 
=& f^{l_{\tau_1}(x)}(x)  + f^{l_{\tau_2}(\tau_1(x))} (\sigma_A^{l_{\tau_1}(x)}(x))  \\
 & -\{  f^{k_{\tau_2}(\tau_1(x))}({\tau_2 \circ\tau_1}(x))
        + f^{k_{\tau_1}(x)} (\sigma_A^{k_{\tau_2}(\tau_1(x) ) }({\tau_2 \circ\tau_1}(x)) )\} \\
=& \rho^f(x,\tau_1) + f^{k_{\tau_1(x)}}(\tau_1(x)) 
     + f^{l_{\tau_2}(\tau_1(x))} (\sigma_A^{k_{\tau_1}(x)}(\tau_1(x))) \\
   &
     - f^{k_{\tau_2}(\tau_1(x))}({\tau_2 \circ\tau_1}(x))
        - f^{k_{\tau_1}(x)} (\sigma_A^{l_{\tau_2}(\tau_1(x) )}(\tau_1(x)) ) \\
 =& \rho^f(x,\tau_1)  + f^{l_{\tau_2}(\tau_1(x)) + k_{\tau_1}(x)}(\tau_1(x)) \\ 
   &  - f^{k_{\tau_1}(x)} (\sigma_A^{l_{\tau_2}(\tau_1(x) )}(\tau_1(x) ) )
        - f^{k_{\tau_2}(\tau_1(x))}({\tau_2 \circ\tau_1}(x)) \\
=& \rho^f(x,\tau_1)  + \rho^f(\tau_1(x), \tau_2).
\end{align*}
Hence $\rho^f$ satisfies the cocycle identity (i).

(ii) As $\rho^f(x,\id) =0,$ the desired identity follows from (i). 
\end{proof}

\begin{definition}
The {\it cocycle full group for}\/ $f$ is defined by 
\begin{equation}\label{eq:GammaAf}
\Gamma_{A,f} = \{ \tau \in \Gamma_A \mid \rho^f(x,\tau) =0 \text{ for all } x \in X_A \}.
\end{equation}
\end{definition}

By Lemma \ref{lem:2.1}, we know the following fact.
\begin{proposition}\label{prop:group}
$\Gamma_{A,f}$ is a subgroup of $\Gamma_A$.
\end{proposition}

\medskip

{\bf 2. Cocycle algebras.}

Let $S_1,\dots, S_N$ be a family of generating partial isometries of $\OA$
satisfying the operator relations
$\sum_{j=1}^N S_j S_j^* =1,\, S_i^* S_i = \sum_{j=1}^N A(i,j) S_j S_j^*, i=1,\dots,N$.
For $f \in C(X_A,\Z)$, recall that the gauge action
$\rho^{A,f}$ with potential $f$ is defined by the automorphisms
$\rho^{A,f}_t, t \in \T$ of $\OA$ satisfying
$$
\rho^{A,f}_t(S_j) = \exp(2 \pi \sqrt{-1} f t) S_j, \qquad t \in \R/ \Z = \T, \, \, j=1,\dots,N.
$$
\begin{definition}
For $f \in C(X_A, \Z)$, define a $C^*$-subalgebra $\FAf$ of $\OA$
by  the fixed point subalgebra of $\OA$ under the action $\rho^{A,f}:$
\begin{equation}
\FAf = 
\{ X \in \OA \mid \rho^{A,f}_t(X) = X \text{ for all } t \in \T\}.
\end{equation}  
We call the $C^*$-algebra $\FAf$ the {\it cocycle algebra for}\/ $f$.
\end{definition}
It is easy to see that 
$\DA \subset \FAf$ for any $f \in C(X_A,\Z)$.

Let us denote by 
$U(\OA,\DA)$ the group of unitary normalizers of $\DA$ in $\OA$:
$$
U(\OA,\DA) = \{ u \in U(\OA) \mid u \DA u^* = \DA\}
$$
where 
$U(\OA)$ denotes the group of unitaries in $\OA$.
The following lemma proved in \cite{MMGGD} is basic in our further discussions,
where the commutative subalgebra $\DA$ is naturally identified with $C(X_A)$.
\begin{lemma}[{\cite{MMGGD}}] \label{lem:fullunitary}
For any $\tau \in \Gamma_A$, there exists 
a unitary $u_\tau \in U(\OA, \DA)$ and
a finite family of admissible words
$\mu(i), \nu(i) \in B_*(X_A), i=1,2,\dots,n$  
satisfying the following conditions:

(1) $u_\tau = \sum_{i=1}^n S_{\mu(i)}S_{\nu(i)}^*$ and
\begin{equation}\label{eq:utauab}
S_{\mu(i)}^* S_{\mu(i)}= S_{\nu(i)}^* S_{\nu(i)}
\quad \text{ and } \quad
\sum_{i=1}^n S_{\mu(i)}S_{\mu(i)}^* 
=\sum_{i=1}^n S_{\nu(i)}S_{\nu(i)}^* =1.
\end{equation}

(2) $u_\tau a u_\tau^* = a\circ \tau^{-1} $ for all $a \in \DA$.  

Conversely, for 
a finite family of admissible words
$\mu(i), \nu(i) \in B_*(X_A), i=1,2,\dots,n$  
satisfying the above conditions \eqref{eq:utauab},
the unitary defined by the formula (1)
 belongs to the unitary normalizer $U(\OA, \DA)$
and there uniquely exists a homeomorphism $\tau$ on $X_A$
belonging to $\Gamma_A$ and satisfying (2). 
\end{lemma}

\begin{lemma}[{\cite[Lemma 3.1]{MaMZ}}] \label{lem:rhoafsmu}
For $\mu \in B_*(X_A)$, the following identity holds:
\begin{equation}
\rho^{A,f}_t(S_\mu) = \exp{(2\pi\sqrt{-1}f^{|\mu|} t)} S_\mu,  \label{eq:rhoafsmu}
\end{equation}
where $|\mu |$ stands for the length $m$ of the word $\mu = (\mu_1,\dots, \mu_m) \in B_m(X_A)$.
\end{lemma}

Define 
$\varphi_\tau(f)\in C(X_A,\Z)$
by setting
\begin{equation*}
\varphi_\tau(f)(x) =  \rho^f(x,\tau),
\qquad x \in X_A.
\end{equation*}

\begin{lemma}\label{lem:rhoaftu}
Let $u \in U(\OA, \DA)$ be a unitary defined by 
$u = \sum_{i=1}^n S_{\mu(i)}S_{\nu(i)}^*$
 for some  admissible words
$\mu(i), \nu(i) \in B_*(X_A), i=1,2,\dots,n$  
satisfying the conditions \eqref{eq:utauab} in Lemma \ref{lem:fullunitary}. 
Let $\tau\in \Gamma_A$ be a homeomorphism defined by (2) in Lemma \ref{lem:fullunitary}.
Let $f \in C(X_A, \Z)$.
Then the following formula holds:
\begin{equation}
\rho^{A,f}_t(u) = \exp{(2 \pi \sqrt{-1} \varphi_{\tau^{-1}}(f) t)} u, \qquad t \in \T. \label{eq:rhoaftu}
\end{equation}
\end{lemma}
\begin{proof}
By Lemma \ref{lem:rhoafsmu}, 
we have
\begin{align*}
\rho^{A,f}_t(S_{\mu(i)} S_{\nu(i)}^*) 
=&  \exp{(2\pi\sqrt{-1}f^{|\mu(i)|} t)} S_{\mu(i)} S_{\nu(i)}^* 
    \exp{(-2\pi\sqrt{-1}f^{|\nu(i)|} t)} \\
=&  \exp{(2\pi\sqrt{-1}f^{|\mu(i)|} t)} S_{\mu(i)} S_{\nu(i)}^* 
    \exp{(-2\pi\sqrt{-1}f^{|\nu(i)|} t)} S_{\nu(i)} S_{\mu(i)}^*
    S_{\mu(i)} S_{\nu(i)}^* \\
=&  \exp{(2\pi\sqrt{-1}f^{|\mu(i)|} t)} \, u \,  
    \exp{(-2\pi\sqrt{-1}f^{|\nu(i)|} t)} \, u^* \, 
    S_{\mu(i)} S_{\nu(i)}^* \\
=&  \exp{(2\pi\sqrt{-1}f^{|\mu(i)|} t)} 
    \exp{(-2\pi\sqrt{-1}f^{|\nu(i)|}\circ \tau^{-1} t)} 
    S_{\mu(i)} S_{\nu(i)}^* \\
=&  \exp{(2\pi\sqrt{-1} ( f^{|\mu(i)|}  - f^{|\nu(i)|}\circ \tau^{-1}) t)} 
    S_{\mu(i)} S_{\mu(i)}^* S_{\mu(i)} S_{\nu(i)}^*. \\
\end{align*}
Now for $x \in U_{\mu(i)}$, we have
$$
|\mu(i)| = k_\tau(x) = l_{\tau^{-1}}(x), \qquad
 |\nu(i)| = l_\tau(x) = k_{\tau^{-1}}(x). 
$$ 
Hence 
$$
f^{|\mu(i)|} (x) - f^{|\nu(i)|}\circ \tau^{-1}(x) =
f^{l_{\tau^{-1}}(x)} (x) - f^{k_{\tau^{-1}}(x)}\circ \tau^{-1}(x) 
= \varphi_{\tau^{-1}}(f)(x)
$$
so that 
\begin{equation*}
\rho^{A,f}_t(S_{\mu(i)} S_{\nu(i)}^*) 
=  \exp{(2\pi\sqrt{-1} \varphi_{\tau^{-1}}(f) t)} 
    S_{\mu(i)} S_{\mu(i)}^* S_{\mu(i)} S_{\nu(i)}^*. 
\end{equation*}
Therefore we obtain that
\begin{equation*}
\rho^{A,f}_t(u) 
=  \sum_{i=1}^n \rho^{A,f}_t(S_{\mu(i)} S_{\nu(i)}^*)  
=  \exp{(2\pi\sqrt{-1} (\varphi_{\tau^{-1}}(f)) t)} u. 
\end{equation*}
\end{proof}
Hence we have the following proposition.
\begin{proposition}\label{prop:rhoafutau}
For $f \in C(X_A,\Z)$ and $\tau \in \Gamma_A$,
we have 
$\rho^{A,f}_t(u_\tau) = u_\tau$
if and only if 
$\tau \in \Gamma_{A,f}.
$
\end{proposition}
\begin{proof}
By Lemma \ref{lem:2.1} (ii), we see that
$\varphi_\tau(f) = -\varphi_{\tau^{-1}}(f)$.
Hence we have
$\tau \in \Gamma_A$ belongs to 
$\Gamma_{A,f}$ if and only if 
$\varphi_{\tau}(f) =0$. 
By \eqref{eq:rhoaftu}, we get the assertion.
\end{proof}


Let us denote by $\Inn(\FAf,\DA)$ the group of inner automorphisms
of $\FAf$ keeping the subalgebra $\DA$ globally.
Let $U(\FAf,\DA)$ be the group of unitary normalizers of $\DA$ 
in $\FAf$:
\begin{equation*}
U(\FAf, \DA) = \{ u \in U(\FAf) \mid u \DA u^* = \DA\}. 
\end{equation*}
For $u \in U(\FAf, \DA)$, $\Ad(u)$ yields an element of $\Inn(\FAf, \DA)$
and defines a homeomorphism $\tau_u$ on $X_A$ satisfying
$\Ad(u)(a) = a \circ \tau_u^{-1}$ for $a \in \DA$.

\begin{theorem}\label{thm:main1}
There exist  short exaxt sequences 
\begin{align}
1\longrightarrow U(\DA)
 \overset{\id}{\longrightarrow}& \, \, U(\FAf,\DA)
 \overset{\Ad}{\longrightarrow} 
\Gamma_{A,f}
 \longrightarrow 1,\label{eq:exactFAf} \\
1\longrightarrow U(\DA)
 \overset{\Ad}{\longrightarrow}& \, \, \Inn(\FAf,\DA)
 \overset{}{\longrightarrow} 
\Gamma_{A,f}
 \longrightarrow 1.\label{eq:InnFAf} 
\end{align}
The latter exact sequence splits.
\end{theorem}
\begin{proof}
Take an arbitrary unitary $u\in U(\FAf,\DA)$.
By \cite[Theorem 1.2]{MaPacific} and \cite[Lemma 2.2]{MMGGD}, 
there exists a finite family of words
$\mu(i), \nu(i) \in B_*(X_A), i=1,2,\dots, n$ 
such that 
 $u_\tau = \sum_{i=1}^n S_{\mu(i)}S_{\nu(i)}^*$ 
and  \eqref{eq:utauab}.
%
%
%
As in \cite{MaPacific}, 
$\Ad: U(\OA,\DA)\longrightarrow \Gamma_A$
gives rise to a surjective homomorphism so that 
there exists $\tau_u \in \Gamma_A$ satisfying 
$\Ad(u)(a) =a\circ\tau_u^{-1}$ for $a \in \DA$.
We will show that $\tau_u \in \Gamma_{A,f}$.
By the previous lemma, we know that 
$$
\rho^{A,f}_t(u) = \exp{(2\pi\sqrt{-1}\varphi_{\tau_u^{-1}}(f) t)} u
$$
Now by the hypothesis that $u \in U(\FAf, \DA)$, we see $u \in \FAf$ and
$
\rho^{A,f}_t(u) =u$.
This implies that 
 $\exp{(2\pi\sqrt{-1}\varphi_{\tau_u^{-1}}(f) t) } =1$ for all $t\in \T$
 and hence
 $\varphi_{\tau_u^{-1}}(f) =0$.
We thus have $\rho^f(x,\tau_u^{-1})=0$ for all $x \in X_A$
so that $\tau_u$ belongs to $\Gamma_{A,f}$.    
Hence the map
$\Ad: u\in U(\FAf, \DA) \longrightarrow \tau_u \in \Gamma_{A,f}$    
defines a homomorphism.
Now assume that $u \in U(\FAf,\DA)$ satisfies $\Ad(u) = \id$ in $\Gamma_{A,f}$. 
We then have $u a u^* =a $ for all $a \in \DA$.
Since $\DA$ is maximal abelian in $\OA$, 
the unitary $u$ belongs to $\DA$.
Since for $\tau \in \Gamma_{A,f}$, 
we have a unitary $u_\tau$ 
as in Lemma \ref{lem:fullunitary} 
together with Proposition \ref{prop:rhoafutau}
such that
$u_\tau \in \FAf$ and 
 $\Ad(u_\tau)(a) = a\circ \tau^{-1}$ for all $a \in \DA$.
Hence
$\Ad: u \in U(\FAf,\DA) \longrightarrow \tau_u \in \in \Gamma_{A,f}$
is surjective, so that we have short exact sequences
\eqref{eq:exactFAf} and \eqref{eq:InnFAf}.
There is a split short exact sequence 
\begin{equation}
1\longrightarrow U(\DA)
 \overset{\Ad}{\longrightarrow} \Inn(\OA,\DA)
 \overset{}{\longrightarrow} 
\Gamma_{A}
 \longrightarrow 1 \label{eq:exactOA} 
\end{equation}
(\cite[Theorem 1.5]{MaPacific}).
As the exact sequence \eqref{eq:InnFAf} is given by restricting 
$\Inn(\OA, \DA)$  to 
 $\Inn(\FAf, \DA)$ in \eqref{eq:exactOA}, 
the exact sequence \eqref{eq:InnFAf} splits.
\end{proof}

\section{Continuous orbit equivalence and cocycle full groups}

\begin{lemma} Let $A, B$ be irreducible, non permutation matrices with entries in $\{0,1\}$.
For $f \in C(X_A,\Z)$ and $g \in C(X_B,\Z)$, 
suppose that there exists an isomorphism 
$\Phi: \FAf\longrightarrow \FBg$ of $C^*$-algebras such that 
$\Phi(\DA) = \DB$.
Then there exists an isomorphism
$\Phi_\Gamma: \Gamma_{A,f} \longrightarrow \Gamma_{B,g}$ 
of groups. 
\end{lemma}
\begin{proof}
In the exact sequence \eqref{eq:exactFAf},
by letting $\Phi |_{U(\DA)}$ the restriction of $\Phi$ to $U(\DA)$,
 we have a commutative diagram at the left square below.
Hence we have a natural isomorphism 
$\Phi_\Gamma: \Gamma_{A,f}\longrightarrow \Gamma_{B,g}$ 
that makes the right square  below commutative:
\begin{equation*}
\begin{CD}
1 @ >>>   U(\DA) @>\id>> U(\FAf,\DA)
  @>\Ad>> \Gamma_{A,f} @>>>1 \\
@. @VV{\Phi |_{U(\DA)}}V
@VV{\Phi}V 
@VV{\Phi_{\Gamma}}V @. \\
1 @ >>>   U(\DB) @>\id>> U(\F_{B,g},\DB)
  @>\Ad>> \Gamma_{B,g} @>>>1. \\
\end{CD}
\end{equation*}
\end{proof}
Now suppose that 
$h:X_A\longrightarrow X_B$ is a homeomorphism 
that gives rise to a continuous orbit equivalence between 
$(X_A,\sigma_A)$ and $(X_B,\sigma_B)$
with functions $k_1, l_1\in C(X_A,\Zp)$
and
$k_2, l_2\in C(X_B,\Zp)$
satisfying
\eqref{eq:coe12}, respectively.
In \cite[(3.2)]{MaJOT2015}, the following  homomorphism
$\Psi_h: C(X_B, \Z)\longrightarrow C(X_A,\Z)$ was introduced: 
\begin{align*}
\Psi_h(g)(x) 
=&  
 \sum_{i=0}^{l_1(x)-1}g(\sigma_B^i(h(x))) 
 \sum_{i=0}^{k_1(x)-1}g(\sigma_B^i(h(\sigma_A(x)))) \\
(=&  
 g^{l_1(x)}(h(x)) - 
 g^{k_1(x)}(h(\sigma_A(x))))\quad \text{ for } g \in C(X_B,\Z), x \in X_A 
\end{align*}
that determines the behavior of the potentials of the gauge actions
under the continuous orbit equivalence in the following way.
\begin{lemma}[{\cite[Theorem 3.3]{MaMZ}}]
Let $h:X_A\longrightarrow X_B$ be a homeomorphism that gives rise to a continuous orbit equivalence between 
$(X_A,\sigma_A)$ and $(X_B,\sigma_B)$.
Then there exists an isomorphism $\OA\longrightarrow \OB$ of $C^*$-algebras 
such that $\Phi(\DA) = \DB$ and 
\begin{equation}
\Phi\circ \rho^{A,\Psi_h(g)}_t = \rho^{B,g}_t\circ \Phi, \qquad t \in \T, \, g\in C(X_B,\Z).
\end{equation}
\end{lemma}
Hence we obtain the following lemma by taking their fixed points under the actions
$\rho^{A,\Psi_h(g)}$ and $\rho^{B,g}$, respectively.
\begin{lemma}
Let $h:X_A\longrightarrow X_B$ be a homeomorphism that gives rise to a continuous orbit equivalence between 
$(X_A,\sigma_A)$ and $(X_B,\sigma_B)$.
Then there exists an isomorphism $\Phi:\OA\longrightarrow \OB$ of 
$C^*$-algebras such that 
$\Phi(\DA) = \DB$ and
\begin{equation}
\Phi(\F_{A,\Psi_h(g)}) = \F_{B,g}, \qquad \text{ for } g \in C(X_B,\Z).
\end{equation}
\end{lemma}
\begin{proposition}\label{prop:GammaUF}
Let $f \in C(X_A,\Z)$ and $g \in C(X_B,\Z)$.
The following assertions are equivalent:
\begin{enumerate}
\renewcommand{\theenumi}{\roman{enumi}}
\renewcommand{\labelenumi}{\textup{(\theenumi)}}
\item There exists a homeomorphism $h:X_A\longrightarrow X_B$ such that 
\begin{equation}
h\circ \Gamma_A \circ h^{-1} = \Gamma_B \quad \text{ and }\quad 
h\circ \Gamma_{A,f} \circ h^{-1} = \Gamma_{B,g}. \label{eq:4.121}
\end{equation}
\item There exists an isomorphism $\Phi:\OA\longrightarrow \OB$ of $C^*$-algebras
such that 
\begin{equation}
\Phi(\DA) = \DB \quad \text{ and } \quad
\Phi(U(\FAf, \DA)) = U(\FBg, \DB). \label{eq:4.12ii}
\end{equation}
\end{enumerate}
\end{proposition}
\begin{proof}
(i) $\Longrightarrow$ (ii):
Assume that there exists a homeomorphism $h:X_A\longrightarrow X_B$
satisfying \eqref{eq:4.121}. 
As in \cite{MaPacific}, the condition $h\circ \Gamma_A \circ h^{-1} = \Gamma_B$
implies that 
$(X_A,\sigma_A)$ and $(X_B,\sigma_B)$ are continuous orbit equivalent.
 Hence there exists an isomorphism $\Phi:\OA\longrightarrow \OB$ of $C^*$-algebras
such that $\Phi(\DA) = \DB$, so that 
$\Phi(U(\OA, \DA)) = U(\OB, \DB).$
It remains to show that $\Phi$ satisfies $\Phi(U(\FAf, \DA)) = U(\FBg, \DB).$ 
For $u \in U(\FAf, \DA)$, we have
$\rho^{A,f}_t(u) = u$ for all $t \in \T$.
Let $\tau = \Ad(u)$ be the homeomorphism on $X_A$ defined by the unitary $u$.
As in the proof of Theorem \ref{thm:main1}, we have 
$\rho^{A,f}_t(u) = \exp{(2 \pi \sqrt{-1} \varphi_{\tau^{-1}}(f) t )} u, t \in \T$,
so that 
$\varphi_{\tau^{-1}}(f) =0$
and hence 
$\rho^f(x,\tau^{-1})=0 $ for all $ x \in X_A$.
This shows that $\tau $ belongs to $\Gamma_{A,f}$.
By the assumption that  
$
h\circ \Gamma_{A,f} \circ h^{-1} = \Gamma_{B,g},
$
we have 
\begin{equation}
\rho^g(y, h\circ \tau \circ h^{-1}) =0 \quad
\text{  for all }
y \in X_B.
\label{eq:4.122}
\end{equation}
Since
$\tau^{-1} = \Ad(u^*)$ on $X_A$, we have
$a\circ \tau = u^* a u$ for all $a \in \DA$. 
As $\Phi(a) = a \circ h^{-1}, a \in \DA$
and
$\Phi(a\circ \tau) = \Phi(u)^* \Phi(a) \Phi(u)$,
we have 
$(a\circ \tau)\circ h^{-1} = \Ad(\Phi(u)^*)(a \circ h^{-1})$.
By putting $b = a\circ h^{-1} \in C(X_B, \Z)$,
we have 
$b \circ (h \circ \tau\circ h^{-1}) = \Ad(\Phi(u)^*)(b)$.
Hence we have
\begin{equation}
h \circ \tau\circ h^{-1} = \Ad(\Phi(u)) \label{eq:4.123}
\end{equation}
By \eqref{eq:4.122}
together with Lemma \ref{lem:2.1} (ii), 
we obtain that 
$\varphi_{h\circ \tau^{-1}\circ h^{-1}}(g) =0$.
Since 
$h\circ \tau^{-1}\circ h^{-1} \in \Gamma_B$,
the unitary $\Phi(u)$ may be written 
as $v\cdot u_0$ where $v \in U(\DA)$ and
$u_0$ is of the form as in Lemma \ref{lem:fullunitary}.
We thus have
 $$
 \rho^{B,g}_t(\Phi(u)) = \exp{(2\pi\sqrt{-1}\varphi_{h\circ \tau^{-1}\circ h^{-1}}(g) t)} \Phi(u) = \Phi(u)
 $$
 so that $\Phi(u) \in \FBg$.
As $\Phi(\DA) = \DB$, we have $\Phi(u)\in U(\FBg, \OB)$.

(ii) $\Longrightarrow$ (i):
Assume that there exists an isomorphism $\Phi:\OA\longrightarrow \OB$ of $C^*$-algebras
satisfying \eqref{eq:4.12ii}.
By \cite{MaPacific}, 
there exists a homeomorphism $h:X_A\longrightarrow X_B$ 
such that $\Phi(a) = a\circ h^{-1}$ and 
$h\circ \Gamma_A \circ h^{-1} = \Gamma_B$.
For $\tau \in \Gamma_{A,f}$, we have $\rho^f(x,\tau) =0.$
Put
$u_\tau = \sum_{i=1}^n S_{\mu(i)}S_{\nu(i)}^* \in U(\OA,\DA)$
the unitary defined by $\tau$ such that
 $\Ad(u_\tau)(a) = a \circ \tau^{-1}$ for $a \in \DA$.
As $\tau \in \Gamma_{A,f}$, one has $\varphi_{\tau^{-1}}(f)=0$. 
Since 
$\rho^{A,f}_t(u_\tau) = \exp{(2\pi\sqrt{-1}\varphi_{\tau^{-1}}(f) t)} u_\tau = u_\tau$,
one obtains $u_\tau \in \FAf$ .
By the hypothesis, one has 
\begin{equation}
\rho^{B,g}_t(\Phi(u_\tau)
) = \Phi(u_\tau). \label{eq:4.125}
\end{equation}
As in \eqref{eq:4.123},
one sees that 
$  \Ad(\Phi(u_\tau))(b) = b\circ (h\circ \tau \circ h^{-1}) $ 
for $b \in \DB$
so that 
\begin{equation}
\rho^{B,g}_t(\Phi(u_\tau)) = \exp{(2\pi\sqrt{-1}\varphi_{h\circ \tau^{-1}\circ h^{-1}}(f) t)}\Phi(u_\tau).
\label{eq:4.126}
\end{equation}
By \eqref{eq:4.125} and \eqref{eq:4.126}, we have
$\varphi_{h\circ \tau^{-1}\circ h^{-1}}(f) =0$.
This implies that 
$\rho^g(y, h \circ \tau\circ h^{-1}) = 0$ for all $y \in X_B$
so that 
$ h \circ \tau\circ h^{-1} \in \Gamma_{B,g}$.
This shows that 
$ h \circ \Gamma_{A,f} \circ h^{-1} \subset \Gamma_{B,g}$
and similarly we have
$ h^{-1} \circ \Gamma_{B,g} \circ h \subset \Gamma_{A,f}$
so that we conclude that 
$ h \circ \Gamma_{A,f} \circ h^{-1} = \Gamma_{B,g}$.
\end{proof}
\begin{remark}
(i) is equivalent to the condition that there exists a homeomorphism 
$h:X_A\longrightarrow X_B$ that gives rise to a continuous orbit equivalence between 
$(X_A,\sigma_A)$ and $(X_B,\sigma_B)$ and  satisfies  
$h\circ \Gamma_{A,f} \circ h^{-1} = \Gamma_{B,g}.$

(ii) If the assertion of (i) holds, then the equality
$f =\Psi_h(g)$ automatically holds.
\end{remark}
Let $h:X_A\longrightarrow X_B$ 
be a homeomorphism that gives rise to a continuous orbit equivalence between 
$(X_A,\sigma_A)$ and $(X_B,\sigma_B)$
with functions $k_1, l_1\in C(X_A,\Zp)$
and
$k_2, l_2\in C(X_B,\Zp)$
satisfying
\eqref{eq:coe12}, respectively.
For $\tau \in \Gamma_A$, we put
$\xi_h(\tau) = h \circ\tau\circ h^{-1}$.  
It has been proved that 
$ \xi_h(\Gamma_A) = \Gamma_B$ (\cite[Proposition 5.4]{MaPacific}).
In the proof of \cite[Proposition 5.4]{MaPacific}, 
we actually showed the following lemma.
\begin{lemma}\label{lem:htauh}
For $\tau \in \Gamma_A$, put
$ m = k_\tau(x), n = l_\tau(x)$. 
Then the equalities
$$
k_{\xi_h(\tau)} (h(x)) = l_1^m(\tau(x)) + k_1^n(x),
\qquad
l_{\xi_h(\tau)} (h(x)) = k_1^m(\tau(x)) + l_1^n(x),
\qquad 
x \in X_A
$$
hold.
This means that 
\begin{equation*}
\sigma_B^{k_{\xi_h(\tau)}(y)}(\xi_h(\tau)(y)) 
= \sigma_B^{l_{\xi_h(\tau)}(y)}(y), \qquad y \in X_B.
\end{equation*}
\end{lemma}
The following proposition is seen in \cite{MaPre2020c}.
We give a proof for the sake of completeness.
\begin{proposition}[{\cite[Proposition 4.4]{MaPre2020c}}]\label{prop:xihg}
Let $h:X_A\longrightarrow X_B$ 
be a homeomorphism that gives rise to a continuous orbit equivalence between 
$(X_A,\sigma_A)$ and $(X_B,\sigma_B)$.
Then there exists an isomorphism $\xi_h:\Gamma_A\longrightarrow \Gamma_B$
such that 
\begin{equation}
\xi_h(\Gamma_{A,\Psi_h(g)}) = \Gamma_{B,g}, \qquad \text{ for } g \in C(X_B,\Z),
\end{equation}
where 
$\xi_h:\Gamma_A\longrightarrow \Gamma_B$ is given by
$\xi_h(\tau) =h\circ \tau\circ h^{-1}, \tau \in \Gamma_A.$
\end{proposition}
\begin{proof}
For a fixed $\tau \in \Gamma_A$, we will show the equality
\begin{equation}
\rho^g(h(x),\xi_h(\tau)) = \rho^{\Psi_h(g)}(x,\tau), \qquad x \in X_A.
\end{equation}
For $\tau \in X_A$,
put $m =k_\tau(x), n= l_\tau(x)$.
By using \cite[Lemma 5.1]{MaPacific} and Lemma \ref{lem:htauh},
we have 
\begin{align*}
\rho^g(h(x),\xi_h(\tau))
= & g^{l_{\xi_{h}(\tau)}(h(x))}(h(x))
   -g^{k_{\xi_{h}(\tau)}(h(x))}(\xi_{h}(\tau)(h(x))) \\
= & g^{k_1^m(\tau(x)) + l_1^n(x)}(h(x))
   -g^{k_1^n(x) + l_1^m(\tau(x))}(h(\tau(x))) \\
= & g^{l_1^n(x)}(h(x)) + g^{k_1^m(\tau(x))}(\sigma_B^{l_1^n(x)}(h(x))) \\
  &- \{ g^{ l_1^m(\tau(x))}(h(\tau(x))) 
       + g^{k_1^n(x)}(\sigma_B^{ l_1^m(\tau(x))}(h(\tau(x)))) \} \\
= & g^{l_1^n(x)}(h(x)) +g^{k_1^m(\tau(x))}(\sigma_B^{k_1^n(x)}(h(\sigma_A^n(x)))) \\
   & - \{ g^{ l_1^m(\tau(x))}(h(\tau(x))) 
       + g^{k_1^n(x)}(\sigma_B^{ k_1^m(\tau(x))}(h(\sigma_A^m(\tau(x))))) \} \\
= & \{ g^{l_1^n(x)}(h(x)) - g^{ l_1^m(\tau(x))}(h(\tau(x))) \} \\
   & + \{ g^{k_1^m(\tau(x))}(\sigma_B^{k_1^n(x)}(h(\sigma_A^n(x)))) 
        - g^{k_1^n(x)}(\sigma_B^{ k_1^m(\tau(x))}(h(\sigma_A^m(\tau(x))))) \}.
 \end{align*}
As
$\sigma_A^n(x) = \sigma_A^m(\tau(x))$,
the second $\{ \quad \}$ above goes to
\begin{align*}
& g^{k_1^m(\tau(x))}(\sigma_B^{k_1^n(x)}(h(\sigma_A^n(x)))) 
        - g^{k_1^n(x)}(\sigma_B^{ k_1^m(\tau(x))}(h(\sigma_A^m(\tau(x))))) \\
= & \{ g^{k_1^n(x)+ k_1^m(\tau(x))}(h(\sigma_A^n(x)))  - g^{k_1^n(x)}(h(\sigma_A^n(x))) \} \\
 & - \{ g^{k_1^n(x) + k_1^m(\tau(x))}(h(\sigma_A^m(\tau(x)))) -g^{k_1^m(\tau(x))}(h(\sigma_A^m(\tau(x))))\} \\
= & g^{k_1^m(\tau(x))}(h(\sigma_A^m(\tau(x)))) - g^{k_1^n(x)}(h(\sigma_A^n(x)))
\end{align*}
 so that we have 
 \begin{align*}
\rho^g(h(x),\xi_h(\tau))
= & \{ g^{l_1^n(x)}(h(x)) - g^{ l_1^m(\tau(x))}(h(\tau(x))) \} \\
   & + \{ g^{k_1^m(\tau(x))}(h(\sigma_A^m(\tau(x)))) - g^{k_1^n(x)}(h(\sigma_A^n(x))) \} \\
= & \{ g^{l_1^n(x)}(h(x)) - g^{k_1^n(x)}(h(\sigma_A^n(x)))   \} \\
   & - \{g^{ l_1^m(\tau(x))}(h(\tau(x))) - g^{k_1^m(\tau(x))}(h(\sigma_A^m(\tau(x)))) \}.
 \end{align*}
By \cite[Lemma 4.5]{MaPAMS2016}, we know
 \begin{align*}
 \varPsi_{h}(g)^{n}(x) 
= & g^{l_1^n(x)}(h(x)) - g^{k_1^n(x)}(h(\sigma_A^n(x))), \\
 \varPsi_{h}(g)^{m}(\tau(x))  
= & g^{l_1^m(\tau(x))}(h(\tau(x))) 
- g^{k_1^m(\tau(x))}(h(\sigma_A^m(\tau(x))))
 \end{align*}
so that  
 \begin{equation*}
\rho^{\varPsi_{h}(g)}(x,\tau)
= \varPsi_{h}(g)^{n}(x) - \varPsi_{h}(g)^{m}(\tau(x)) = \rho^{\varPsi_h(g)}(x,\tau). 
  \end{equation*}
\end{proof}


\section{Cocycle full groups and one-sided conjugacy} \label{sec:final}
In this section, we are assuming that 
there exists a homeomorphism $h:X_A\longrightarrow X_B$ that gives rise to 
a continuous orbit equivalence between $(X_A, \sigma_A)$ and $(X_B, \sigma_B)$.
\begin{lemma}\label{lem:cocyclegroup1}
Assume that 
\begin{equation}
\Gamma_{A, \Psi_h(g)} = \Gamma_{A, g\circ h}
\quad \text{ for all }
g \in C(X_B,\Z).
\end{equation}
Then we have 
\begin{equation} \label{eq:Agh}
\Gamma_{A, g\circ h - g \circ h \circ \sigma_A}
= \Gamma_{A, g\circ h - g \circ \sigma_B \circ h}
\quad \text{ for all }
g \in C(X_B,\Z).
\end{equation}
\end{lemma}
\begin{proof}
By \cite[Lemma 4.6]{MMETDS},
we have
$\Psi_h(g - g\circ \sigma_B) = g\circ h - g \circ h \circ \sigma_A$, 
proving \eqref{eq:Agh}.
\end{proof}
\begin{lemma}\label{lem:cocyclegroup2}
For $g \in C(X_B,\Z)$ and $\tau \in \Gamma_A$, we have 
\begin{enumerate}
\renewcommand{\theenumi}{\roman{enumi}}
\renewcommand{\labelenumi}{\textup{(\theenumi)}}
\item
$\tau \in \Gamma_{A, g\circ h - g \circ h \circ \sigma_A}$
if and only if $g(h(x)) = g(h(\tau(x)))$ for all $x \in X_A$.
\item
$\tau \in \Gamma_{A, g\circ h - g \circ \sigma_B \circ h}$
if and only if
\begin{align} 
& \sum_{i=0}^{l_\tau(x)}g(h(\sigma_A^i(x))) -
\sum_{i=0}^{l_\tau(x)}g(\sigma_B(h(\sigma_A^i(x)))) \notag \\
= &
\sum_{i=0}^{k_\tau(x)}g(h(\sigma_A^i(\tau(x)))) -
\sum_{i=0}^{k_\tau(x)}g(\sigma_B(h(\sigma_A^i(\tau(x))))) \label{eq:AgsB}
\end{align}
 holds for all $x \in X_A$.
\end{enumerate}
\end{lemma}
\begin{proof}
It is straightforward to see that 
$\tau \in \Gamma_{A, g\circ h - g \circ h \circ \sigma_A}$
if and only if 
the equality 
 \begin{equation} \label{eq:AgsB2}
\sum_{i=0}^{l_\tau(x)}g(h(\sigma_A^i(x))) -
\sum_{i=0}^{l_\tau(x)}g(h(\sigma_A(\sigma_A^i(x))))
=
\sum_{i=0}^{k_\tau(x)}g(h(\sigma_A^i(\tau(x)))) -
\sum_{i=0}^{k_\tau(x)}g(h(\sigma_A(\sigma_A^i(\tau(x)))))
\end{equation}
 holds for all $x \in X_A$.
The above equality 
is nothing but the equality
$g(h(x)) = g(h(\tau(x)))$ for all $x \in X_A$.

(ii) 
Since
$\tau \in \Gamma_{A, g\circ h - g \circ \sigma_B \circ h}$
is equivalent to the equality
$\rho^{g\circ h - g \circ \sigma_B \circ h}(x, \tau) =0$ 
for all $x \in X_A$.
The latter is equivalent to 
the equality \eqref{eq:AgsB} for all $x \in X_A$.
\end{proof}

\begin{proposition}\label{prop:cocyclegroup}
Let $h:X_A\longrightarrow X_B$ be
a homeomorphism that gives rise to 
a continuous orbit equivalence between $(X_A, \sigma_A)$ and $(X_B, \sigma_B)$.
Assume that 
\begin{equation} \label{eq:GammaAP}
\Gamma_{A, \Psi_h(g)} = \Gamma_{A, g\circ h}
\quad \text{ for all }
g \in C(X_B,\Z).
\end{equation}
Then we have $h \circ \sigma_A = \sigma_B\circ h$.
This means that 
$(X_A, \sigma_A)$ and $(X_B, \sigma_B)$
are topologically conjugate.
\end{proposition}
\begin{proof}
Assume that the equality \eqref{eq:GammaAP}
holds for all $g \in C(X_B,\Z)$.
By Lemma \ref{lem:cocyclegroup1},
the equality \eqref{eq:Agh} holds for all $g \in C(X_B,\Z)$.
Let $X_A^{\operatorname{nep}}$ 
be the set of non eventually periodic points of 
$\sigma_A$ in $X_A$.
That is,
\begin{equation*}
X_A^{\operatorname{nep}}
=\{ x \in X_A \mid \sigma_A^k(x) \ne \sigma_A^l(x) \text{ for all } k,l \in \Zp
\text{ with } k\ne l\}.
\end{equation*}
As the matrix $A$ is irreducible and not any permutation,
the set 
$X_A^{\operatorname{nep}}$ is dense in $X_A$.

Now assume that 
$h \circ \sigma_A \ne \sigma_B\circ h$.
Since 
$X_A^{\operatorname{nep}}$ is dense in $X_A$, 
one may find $z \in X_A^{\operatorname{nep}}$ such that 
\begin{equation} \label{eq:prop1}
h \circ \sigma_A(z) \ne \sigma_B\circ h (z).
\end{equation}
As the set of non eventually periodic points is
preserved by a homeomorphism giving rise to a continuous orbit 
equivalence (\cite[Proposition 3.5]{MMETDS}),
we have
\begin{align}
\sigma_A(z) & \ne z, \label{eq:prop2} \\
\sigma_B(h(z)) & \ne h(z). \label{eq:prop3} 
\end{align}
By taking higher block representation of
$(X_A,\sigma_A)$ (cf. \cite{LM}), 
we may assume that by \eqref{eq:prop2}
for the $z = (z_n)_{n\in \N}$, 
\begin{equation}
z_1 \ne z_2. \label{eq:prop3}
\end{equation}
Since
\begin{equation}
h(\sigma_A(z)) \ne \sigma_B(h(z)), \qquad 
h(z) \ne \sigma_B(h(z)), \label{eq:prop5}
\end{equation}
there exists a finite word $\mu =(z_1,\dots,z_m) \in B_m(X_A)$
such that 
\begin{equation*}
h(\sigma_A(x)) \ne \sigma_B(h(z)), \qquad 
h(x) \ne \sigma_B(h(z)) \qquad \text{ for all } x \in U_\mu.
\end{equation*}
By taking higher block representation of
$(X_A,\sigma_A)$, we may assume that 
$\mu = (z_1,z_2)\in B_2(X_A)$ and 
\begin{equation*}
z_1\ne z_2, \qquad h(\sigma_A(x)) \ne \sigma_B(h(z)), \qquad 
h(x) \ne \sigma_B(h(z)) \qquad \text{ for all } x \in U_{z_1 z_2}.
\end{equation*}
This means that 
\begin{equation*}
\sigma_B(h(z)) \not\in \{ h(\sigma_A(x)), h(x) \mid x \in U_{z_1 z_2} \}. 
\end{equation*}
Hence we may find $g \in C(X_B,\Z)$ such that 
\begin{align}
g(h(\sigma_A(x))) & = g(h(x)) = 0 \qquad \text{ for all } x \in U_{z_1 z_2}, \label{eq:4.12}\\ 
g(\sigma_B(h(z))) & =1. \label{eq:4.13}
\end{align}
By \cite[Lemma 3.2]{MaPacific}, there exists
$\tau_0 \in \Gamma_A$ such that for $x \in X_A$
\begin{equation*}
\tau_0(x) =
\begin{cases}
\sigma_A (x) \in U_{z_2}     & \text{ if } x \in U_{z_1 z_2},\\
z_1 x \in U_{z_1 z_2}               & \text{ if } x \in U_{z_2}, \\
x & \text{ otherwise.}
\end{cases}
\end{equation*}
We set
\begin{equation*}
k_{\tau_0}(x) =
\begin{cases}
0    & \text{ if } x \in U_{z_1 z_2}, \\
1  & \text{ if } x \in  U_{z_2}, \\
0 & \text{ otherwise,}
\end{cases}
\qquad
l_{\tau_0}(x) =
\begin{cases}
1    & \text{ if } x \in U_{z_1 z_2},\\ 
0  & \text{ if } x \in  U_{z_2},\\
0 & \text{ otherwise}
\end{cases}
\end{equation*}
so that
\begin{equation*}
\sigma_A^{k_{\tau_0}(x)}(\tau_0(x))
  = \sigma_A^{l_{\tau_0}(x)}(x)\qquad \text{ for } x \in X_A.
\end{equation*}
For $x \in X_A$, we have three cases.

Case 1: $x \in U_{z_1 z_2}$: 
We have $\tau_0(x) = \sigma_A(x)$. 
By \eqref{eq:4.12}, 
the equality $g(h(x)) = g(h(\tau_0(x))) =0$
holds.

Case 2: $x \in U_{z_2}$: 
We have $\tau_0(x) = z_1 x$. 
Put $x' = z_1 x$, 
so that  $x' \in U_{z_1 z_2}, \sigma_A(x') = x$.
By \eqref{eq:4.12}, 
the equality $g(h(x')) = g(h(\tau_0(x'))) =0$
holds, so that we have
$g(h(\tau_0(x))) = g(h(x)) =0.$

Case 3: $x \not\in U_{z_1 z_2}\cup U_{z_2}$: 
We have $\tau_0(x) = x$. 
By \eqref{eq:4.12}, 
the equality $g(h(\tau_0(x))) = g(h(x))$
holds.

Consequently we obtain that 
\begin{equation*}
g(h(x)) = g(h(\tau_0(x))) \qquad \text{ for all } x \in X_A
\end{equation*}
so that by Lemma \ref{lem:cocyclegroup2}, we have
\begin{equation*}
\tau_0 \in \Gamma_{A, g\circ h - g\circ h \circ \sigma_A}.
\end{equation*}
By the hypothesis $\Gamma_{A,\Psi_h(g)} = \Gamma_{A,g\circ h}$ 
with Lemma \ref{lem:cocyclegroup1}, we have 
\begin{equation*}
\tau_0 \in \Gamma_{A, g\circ h - g\circ \sigma_B \circ h}.
\end{equation*}
Since
$z \in U_{z_1 z_2}$, 
we know that 
$\tau_0(z) = \sigma_A(z), \, l_{\tau_0}(z) = 1,\, k_{\tau_0}(z) = 0.$
By \eqref{eq:4.12} and \eqref{eq:4.13},  we have 
\begin{align*}
 & \sum_{i=0}^{l_{\tau_0}(z)} g(h(\sigma_A^i(z))) -
    \sum_{i=0}^{l_{\tau_0}(z)} g(\sigma_B(h(\sigma_A^i(z))))) \\
=&\{g(h(z))+ g(h(\sigma_A(z))) \} -\{ g(\sigma_B(h(z)))-  g(\sigma_B(h(\sigma_A(z))) ) \} \\
=& g(\sigma_B(h(\sigma_A(z))) ) -1
\end{align*}
and
\begin{align*}
 & \sum_{i=0}^{k_{\tau_0}(z)} g(h(\sigma_A^i(\tau_0(z)))) -
    \sum_{i=0}^{k_{\tau_0}(z)} g(\sigma_B(h(\sigma_A^i(\tau_0(z))))) \\
=& g(h(\tau_0(z)))- g(\sigma_B(h(\tau_0(z)))) \\
=& g(h(\sigma_A(z))) -  g(\sigma_B(h(\sigma_A(z))) ) \\
=&  -  g(\sigma_B(h(\sigma_A(z))) ).
\end{align*}
By Lemma \ref{lem:cocyclegroup2} (ii), 
the equality \eqref{eq:AgsB} holds so that we have 
$
g(\sigma_B(h(\sigma_A(z))) ) -1 = -g(\sigma_B(h(\sigma_A(z))) ) 
$
a contradiction because 
$
g(\sigma_B(h(\sigma_A(z))) ) \in \Z. 
$
Consequently the hypothesis 
\eqref{eq:prop1} yields the contradiction so that we conclude
$h\circ \sigma_A = \sigma_B\circ h$.
\end{proof}

We provide a proposition which is seen in \cite{MaPre2020c}.
We give its proof for the sake of completeness.
\begin{proposition}[{\cite[Proposition 4.2]{MaPre2020c}}] \label{prop:uniqueh}
Assume that there exists an isomorphism
$\xi: \Gamma_A \longrightarrow \Gamma_B$ as abstract groups.
Then there exists a unique homeomorphism $h:X_A\longrightarrow X_B$ 
such that $\xi(\tau) = h\circ \tau\circ h^{-1}, \tau \in \Gamma_A.$  
\end{proposition}
\begin{proof}
By \cite[Theorem 7.2]{MaIsrael2015},  
there exists a homeomorphism
$h:X_A\longrightarrow X_B$ 
such that $\xi(\tau) = h\circ \tau\circ h^{-1}, \tau \in \Gamma_A.$
It remains to show its uniqueness.
Suppose that there exists another  
homeomorphism
$h_1:X_A\longrightarrow X_B$ 
satisfying $\xi(\tau) = h_1\circ \tau\circ h_1^{-1}, \tau \in \Gamma_A.$
We put $h_0 = h^{-1}\circ h_1$ that is a homeomorphism 
on $X_A$ commuting  with all elements of $\Gamma_A$.
We are assuming that $h_0 \ne \id$, so that 
one may find admissible words $\mu, \nu \in B_*(X_A)$ 
such that their  cylinder sets $U_\mu, U_\nu \subset X_A$
are disjoint and  satisfy 
\begin{equation*}
h_0(U_\mu) \subset U_\nu, \qquad
U_\nu \backslash h_0(U_\mu) \ne \emptyset.
\end{equation*}
Put $U = h_0(U_\mu), Y = U_\nu \backslash h_0(U_\mu)$
that are open sets in $X_A$
and take $z \in U$.
By \cite[Lemma 2.1]{MaIsrael2015}, 
one may find a clopen set $V$ of $X_A$ and an element 
$\tau_0 \in \Gamma_A$
such that 
\begin{equation*}
z \in V \subset U, \qquad 
\tau_0(V) \subset Y, \qquad 
\tau_0|_{(V\cup\tau(V))^c} =\id.
\end{equation*}
Since  we have
$h_0^{-1}(z) \in U_\mu$ and
$V \cup \tau_0(V) \subset U_\nu$, we see that 
$
h_0^{-1}(z) \in (V\cup\tau_0(V))^c,
$
because 
$U_\nu \cap U_\mu = \emptyset$.
Hence we have $\tau_0(h_0^{-1}(z)) = h_0^{-1}(z)$, 
so that   
\begin{equation*}
(h_0\circ \tau_0)(h_0^{-1}(z))= z \in V \subset h_0(U_\mu).
\end{equation*}
On the other hand, we see that 
$
(\tau_0\circ h_0)(h_0^{-1}(z))= \tau_0(z) \in \tau_0(V) \subset U_\nu\backslash h_0(U_\mu).
$
We thus obtain that 
\begin{equation*}
(h_0\circ \tau_0)(h_0^{-1}(z))\ne (\tau_0\circ h_0)(h_0^{-1}(z))
\end{equation*}
so that  
$h_0\circ \tau_0 \ne \tau_0\circ h_0$, a contradiction.
We thus have $h_0 =\id$ and hence $h_1 = h$.
\end{proof}

\begin{theorem}\label{thm:main4}
Let $A, B$ be irreducible non permutation matrices with entries in $\{0,1\}$.
If there exists an isomorphism
$\xi: \Gamma_A\longrightarrow \Gamma_B$ of groups such that 
$\xi(\Gamma_{A,g\circ h}) =\Gamma_{B,g}$ for all $g \in C(X_B,\Z)$,
then $h\circ \sigma_A = \sigma_B\circ h$,
where $h:X_A\longrightarrow X_B$
is a unique homeomorphism satisfying 
$\xi(\tau) = h\circ \tau\circ h^{-1}$ for $\tau \in \Gamma_A$.
Hence 
$(X_A, \sigma_A)$ and $(X_B, \sigma_B)$
are topologically conjugate.
\end{theorem}
\begin{proof}
By \cite[Proposition 4.4]{MaPre2020c},
the identity
$\xi_h(\Gamma_{A,\Psi_h(g)}) = \Gamma_{B,g}$ holds for all
$g \in C(X_B,\Z)$.
By the hypothesis
$\xi(\Gamma_{A,g\circ h}) =\Gamma_{B,g}$, we have
$\Gamma_{A,\Psi_h(g)}=\Gamma_{A,g\circ h}.$
Therefore by Proposition \ref{prop:cocyclegroup}, we conclude that 
$h \circ \sigma_A = \sigma_B\circ h.$
\end{proof}

Now we may finally complete the proof of Theorem \ref{thm:mainthm}.

\medskip

{\it Proof of Theorem \ref{thm:mainthm}.}

The equivalence 
(i) $\Longleftrightarrow$ (ii) was proved in \cite{MaPre2020b}.

The implication 
(ii) $\Longrightarrow$ (iii) is obvious.

Assume the condition (iii).
Hence we have 
$\Phi(U(\F_{A,g\circ h}))= U(\F_{B,g})$ for all $g \in C(X_B,\Z)$.
By Proposition \ref{prop:GammaUF} (ii) $\Longrightarrow$ (i),
we know that the assertion (iv) holds.

The implication (iv) $\Longrightarrow$ (i)
follows from Theorem \ref{thm:main4}.
\medskip

In \cite{MaPre2020d},
we will construct an \'etale groupoid $G_{A,f}$ such that 
$C^*(G_{A,f}) = \F_{A,f}$ and describe the condition (iii)
in terms of the groupoids.

\medskip

{\it Acknowledgment:}
This work was supported by JSPS KAKENHI 
Grant Numbers 19K03537.


\begin{thebibliography}{99}

\bibitem{BC}
{\sc K. A. Brix and T. M. Carlsen}
{\it Cuntz-Krieger algebras and one-sided conjugacy of shifts of finite type 
and their groupoids}, 
J. Aust. Math. Soc. (2019), 1-10, doi:10:1017/S1446788719000168,
arXiv:1712.00179 [mathOA].


\bibitem{CEOR}
{\sc T. M. Carlsen,  S. Eilers, E. Ortega and G. Restorff},
{\it Flow equivalence and orbit equivalence for shifts of finite type 
and isomorphism of their groupoids}, 
J. Math. Anal. Appl. {\bf 469}(2019), pp. \ 1088--1110. 


\bibitem{CRS}
{\sc T. M. Carlsen, E. Ruiz and A. Sims},
{\it Equivalence and stable isomorphism of groupoids,
and diagonal-preserving stable isomorphisms of graph $C^*$-algebras 
and Leavitt path algebras},
Proc. Amer. Math. Soc. {\bf 145}(2017), pp. \ 1581--1592.


\bibitem{CK}{\sc J. ~Cuntz and W. ~Krieger},
{\it A class of $C^*$-algebras and topological Markov chains},
 Invent.\ Math.\
 {\bf 56}(1980), pp.\ 251--268.

\bibitem{Ki}{\sc B.~P. ~Kitchens},
{\it Symbolic dynamics}, 
Springer-Verlag, Berlin, Heidelberg and New York
(1998).


\bibitem{LM}{\sc D. ~Lind and B. ~Marcus},
{\it An introduction to symbolic dynamics and coding},
 Cambridge University Press, Cambridge
(1995).

\bibitem{MaPacific}
{\sc K. Matsumoto},
{\it Orbit equivalence of topological Markov shifts and Cuntz-Krieger algebras},
Pacific J.\ Math.\ 
{\bf 246}(2010),
pp.\  199--225.


\bibitem{MaDCDS2013}
{\sc K. Matsumoto},
{\it K-groups of the full group actions on one-sided topological Markov shifts},
 Discrete and Contin. Dyn. Syst.
{\bf 33}(2013),  pp.\ 3753--3765.





\bibitem{MaIsrael2015}{\sc K. Matsumoto},
{\it Full groups of one-sided topological Markov shifts},
Israel J. Math.
{\bf 205}(2015), pp. \ 1--33.

\bibitem{MaJOT2015}
{\sc K. Matsumoto},
{\it Strongly continuous orbit equivalence of 
one-sided topological Markov shifts},
J. Operator Theory {\bf 74}(2015), pp. 101--127.

%
\bibitem{MaPAMS2016}
{\sc K. Matsumoto},
{\it  On flow equivalence of one-sided topological Markov shifts},
Proc. Amer. Math. Soc. {\bf 144}(2016), \ 2923--2937.%

\bibitem{MaPAMS2017}{\sc K. Matsumoto},
{\it Uniformly continuous orbit equivalence of Markov shifts 
and gauge actions on Cuntz--Krieger algebras},
Proc. Amer. Math. Soc. {\bf 145}(2017), pp.\ 1131--1140. 

\bibitem{MaMZ}
{\sc K. Matsumoto},
{\it Continuous orbit equivalence, 
flow equivalence of Markov shifts and circle actions on Cuntz--Krieger algebras},
 Math. Z. {\bf 285}(2017), \ 121--141.  





\bibitem{MaPre2020b}
{\sc K. Matsumoto},
{\it On one-sided topological conjugacy of topological Markov shifts 
and gauge actions on Cuntz--Krieger algebras},
preprint,  arXiv:2007.01974v1 [mathOA].

\bibitem{MaPre2020c}
{\sc K. Matsumoto},
{\it Cohomology groups, continuous full groups and  ontinuous orbit equivalence
of topological Markov shifts},
 preprint, arXiv:2012.11598 [math.DS].  

\bibitem{MaPre2020d}
{\sc K. Matsumoto},
{\it On a family of $C^*$-subalgebras of Cuntz--Krieger algebras},
 preprint.  

\bibitem{MMKyoto}
{\sc K. Matsumoto and H. Matui},
{\it Continuous orbit equivalence of topological Markov shifts 
and Cuntz--Krieger algebras},
Kyoto J. Math. {\bf 54}(2014), pp.\ 863--878.


\bibitem{MMETDS}
{\sc K. Matsumoto and H. Matui},
{\it Continuous orbit equivalence of topological Markov shifts 
and dynamical zeta functions}, 
Ergodic Theory Dynam. Systems
{\bf 36}(2016), pp. \ 1557--1581.


\bibitem{MMGGD}
{\sc K. Matsumoto and H. Matui},
{\it Full groups of Cuntz-Krieger algebras and Higman-Thompson groups},
 Groups Geom. Dyn. {\bf 11}(2017), pp. \ 499--531.

\bibitem{MatuiPLMS}{\sc H. Matui}, 
{\it Homology and topological full groups of {\'e}tale groupoids on totally disconnected spaces},
Proc. London Math. Soc. {\bf 104}(2012),  pp.\ 27--56.


\bibitem{MatuiCrelle}{\sc H. Matui}, 
{\it Topological full groups of one-sided shifts of finite type},
J. Reine Angew. Math. {\bf 705}(2015), pp. \ 35--84.



\bibitem{NO}{\sc P. Nyland and E. Ortega},
{\it Topological full groups of ample groupoids with applications to graph algebras},
 Internat. J. Math., {\bf 30}( 2019),  1950018, \, 66 pages.



\bibitem{Renault}{\sc J. Renault},
{\it A groupod approach to $C^*$-algebras},
Lecture Notes in Math.  793 Springer.


\bibitem{Renault2008}{\sc J. Renault},
{\it Cartan subalgebras in $C^*$-algebras},
Irish Math. Soc. Bull.{\bf 61}(2008), pp. 29--63.

\end{thebibliography}
\end{document}